\documentclass{article}
\usepackage[utf8]{inputenc}
\usepackage[T1]{fontenc}
\usepackage{amsthm, amssymb, amsmath}
\usepackage{xcolor}
\usepackage{authblk}
\def\margin{3.9cm}
\usepackage[left=\margin,right=\margin,top=\margin,bottom=\margin]{geometry}

\title{Conflict-free coloring on open neighborhoods of claw-free graphs}
\author[1]{Sriram Bhyravarapu\thanks{This work was done when the author was a Ph.D. student at Department of CSE, IIT Hyderabad.}}
\author[2]{Subrahmanyam Kalyanasundaram}
\author[2]{Rogers Mathew}
\affil[1]
{
The Institute of Mathematical Sciences, HBNI, Chennai, India \authorcr
{\tt sriramb@imsc.res.in}
}
\affil[2]
{
Department of Computer Science and Engineering, \authorcr
Indian Institute of Technology Hyderabad, India - 502285. \authorcr {\tt \{subruk, rogers\}@cse.iith.ac.in}
}
\date{}

\theoremstyle{definition}

\theoremstyle{plain}
\newtheorem{theorem}{Theorem}
\newtheorem{lemma}[theorem]{Lemma}
\newtheorem{corollary}[theorem]{Corollary}

\newtheorem{observation}[theorem]{Observation}

\newtheorem{definition}[theorem]{Definition}

\theoremstyle{remark}

\begin{document}

\maketitle
\begin{abstract}
The \emph{Conflict-Free Open (Closed) Neighborhood coloring}, abbreviated \emph{CFON (CFCN) coloring}, of a graph $G$ using $r$ colors is a coloring of the vertices of $G$ such that every vertex sees some color exactly once in its open (closed) neighborhood. The minimum $r$ such that $G$ has a CFON (CFCN) coloring using $r$ colors is called the \emph{CFON chromatic number} (\emph{CFCN chromatic number}) of $G$. This is denoted by $\chi_{CF}^{ON}(G)$ ($\chi_{CF}^{CN}(G)$). D\k ebski and Przyby\l{}o in 
[J. Graph Theory, 2021] 
showed that if $G$ is a line graph with maximum degree $\Delta$, then $\chi_{CF}^{CN}(G) = O(\ln \Delta)$. As an open question, they
asked if the result could be extended to claw-free ($K_{1,3}$-free) graphs, which are a superclass of line graphs. 
For $k\geq 3$, we show that if $G$ is $K_{1,k}$-free, then $\chi_{CF}^{ON}(G) = O(k^2\ln \Delta)$. Since it is known that the CFCN chromatic number of a graph is at most twice its CFON chromatic number, this answers the question posed by  D\k{e}bski and Przyby\l{}o.    
\end{abstract}
\section{Introduction}
\begin{definition}[Conflict-free coloring]
Given a hypergraph $\mathcal{H}= (V, \mathcal{E})$, a coloring $c~:V\rightarrow [r]$ is a \emph{conflict-free coloring} of $\mathcal{H}$ if for every hyperedge $E \in \mathcal{E}$, there is a vertex in $E$ that receives a color under $c$ that is distinct from the colors received by all the other vertices in   $E$. The minimum $r$ such that $c~:V\rightarrow [r]$ is a conflict-free coloring of $\mathcal{H}$ is called the \emph{conflict-free chromatic number} of $\mathcal{H}$. This is denoted by $\chi_{CF}(\mathcal{H})$. 
\end{definition}
For a graph $G$, we shall use $V(G)$ and $E(G)$ to denote its vertex set and edge set, respectively. For a vertex $v$ in $G$, its \emph{open neighborhood}, denoted by $N_G(v)$, is defined as the set of neighbors of $v$ in $G$. The \emph{closed neighborhood} of $v$ is $N_G(v) \cup \{v\}$. 
\begin{definition}[CFON coloring]
A coloring $c:V(G) \rightarrow [r]$ of the vertices of $G$ is called a \emph{Conflict-Free Open Neighborhood} coloring (or CFON coloring) of $G$ if for every vertex in $G$, there is some color that appears exactly once in its open neighborhood. The minimum $r$ such that a coloring $c:V(G) \rightarrow [r]$ is a CFON coloring of $G$ is called the \emph{Conflict-Free Open Neighborhood chromatic number} (or CFON chromatic number) of $G$. This is denoted by $\chi_{CF}^{ON}(G)$.   
\end{definition}
\begin{definition}[CFCN coloring]
A coloring $c:V(G) \rightarrow [r]$ of the vertices of $G$ is called a \emph{Conflict-Free Closed  Neighborhood} coloring (or CFCN coloring) of $G$ if for every vertex in $G$, there is some color that appears exactly once in its closed neighborhood. The minimum $r$ such that a coloring $c:V(G) \rightarrow [r]$ is a CFCN coloring of $G$ is called the \emph{Conflict-Free Closed Neighborhood chromatic number} (or CFCN chromatic number) of $G$. This is denoted by $\chi_{CF}^{CN}(G)$.   
\end{definition}
Since its introduction by Even, Lotker, Ron and Smorodinsky in 2004 \cite{Even2002}, conflict-free coloring has been extensively studied \cite{smorosurvey, SSR_JGT, Pach2009, glebov2014conflict}. 
\begin{definition}(Claw-free graph)
The complete bipartite graph $K_{1,3}$ is called a \emph{claw}.  A graph is called a  \emph{claw-free graph} if it does not contain a claw as an induced subgraph. 
\end{definition}
In this paper, we also consider $K_{1,k}$-free graphs, which are similarly defined.
The following result is due to D\k ebski and Przyby\l{}o \cite{DebskiPrzyblo}.
\begin{theorem}[\cite{DebskiPrzyblo}]
\label{thm_Debski}
Let $G$ be a line graph with maximum degree $\Delta$. Then $\chi_{CF}^{CN}(G) = O(\ln \Delta)$. This bound is tight up to constants. 
\end{theorem} 
The line graph of a complete graph was used to show the tightness of above upper bound.  Line graphs are a subclass of claw-free graphs. In \cite{DebskiPrzyblo}, it was asked whether the above result can be extended to claw-free graphs. The main contribution of this paper is to answer this question by extending the above result to claw-free graphs, by showing the following bound. 
\begin{theorem}
\label{thm_main}
Let $G$ be a $K_{1,k}$-free graph with maximum degree $\Delta$ having no isolated vertices. Then, $\chi_{CF}^{ON}(G) = O(k^2\ln\Delta)$.  
\end{theorem}
It is known (see for instance, equation 1.3 from \cite{Pach2009}) that the CFCN chromatic number of a graph with no isolated vertices is at most twice its CFON chromatic number. 
If there are isolated vertices, we can give them a separate new color. 
We thus have the following corollaries.
\begin{corollary}
\label{cor_closed_neighborhood}
Let $G$ be an $K_{1,k}$-free graph with maximum degree $\Delta$. Then, $\chi_{CF}^{CN}(G) = O(k^2\ln\Delta)$. 
\end{corollary}

\begin{corollary}
\label{cor_claw_free}
Let $G$ be a claw-free ($K_{1,3}$-free) graph with no isolated vertices. Then, \\ (i) $\chi_{CF}^{CN}(G) = O(\ln\Delta)$, and (ii) if $G$ has no isolated vertices, $\chi_{CF}^{ON}(G) = O(\ln\Delta)$. 
\end{corollary}

\section{Preliminaries}
Below we state the Local Lemma due to Erd\H{o}s and Lov{\'a}sz \cite{lovaszlocallemma} that will be used in some of the proofs in  the paper. 
\begin{lemma}[\emph{The Local Lemma}, \cite{lovaszlocallemma}] \label{lem:local} Let $A_1, \ldots , A_n$ be events in an arbitrary probability space. Suppose that each event $A_i$ is mutually independent of a set of all the other events $A_j$ but at most $d$, and that $Pr[A_i] \leq p$ for all $i \in [n]$. If  $$4pd \leq 1,$$ then $Pr[\cap _{i=1}^n \overline{A_i}] > 0$.  
\end{lemma}
Below we state a version of the Talagrand's Inequality from \cite{MolloyR14}.
\begin{theorem}[Talagrand's Inequality, \cite{MolloyR14}]
\label{thm_Talagrand}
Let $X$ be a non-negative random variable determined by the independent trials $T_1, \ldots , T_n$. Suppose that for every set of possible outcomes of the trials, we have:
\\(i) changing the outcome of any one trial can affect $X$ by at most $a$; and
\\(ii) for each $s>0$, if $X \geq s$ then there is a set of at most $bs$ trials whose outcomes certify that $X \geq s$.
\\Then for any $t \geq 0$, we have
$$Pr[|X - E[X]| > t + 20a\sqrt{b E[X]} + 64a^2b] \leq 4e^{-\frac{t^2}{8a^2b(E[X] + t)}}.$$
\end{theorem}
The following theorem on conflict-free coloring of hypergraphs is from \cite{Pach2009}. The degree of a vertex in a hypergraph is the number of hyperedges it is part of. 
\begin{theorem}[Theorem 1.1(b) in \cite{Pach2009}]
\label{thm_pach_tardos}
Let $\mathcal{H}$ be a hypergraph and let $\Delta$ be the maximum degree of any vertex in $\mathcal{H}$. Then, $\chi_{CF}(\mathcal{H}) \leq \Delta + 1$.  
\end{theorem}

\section{Proof of our results}
\begin{lemma}
\label{lem_near_uniform_hypergraph}
Let $\mathcal{H} = (V,\mathcal{E})$ be a hypergraph where (i) every hyperedge intersects with at most $\Gamma$ other hyperedges, and (ii) for every hyperedge $E \in \mathcal{E}$, $r \leq |E| \leq cr$, where $c \geq 1$ is some integer and $r = \max(2^{12}, \lceil 136 \ln(16\Gamma)\rceil)$. Then, 
$$\chi_{CF}(\mathcal{H}) \leq 32cr.$$
\end{lemma}
\begin{proof}
For each vertex in $V$, assign a color that is chosen independently and uniformly at random from a set of $32cr$ colors. We will first use 
Talagrand's Inequality to show that the probability of this coloring being bad for an edge is small, and then use Local Lemma to show the existence of conflict-free coloring for $\mathcal H$ using at most $32cr$ colors.

Consider a hyperedge $E \in \mathcal{E}$. Let $X_E$ be a random variable that denotes the number of vertices in $E$ whose color is also assigned to some other vertex in $E$.  For a vertex in $E$, the probability that its color is also assigned to some other vertex in $E$ is at most $1 - (1 - \frac{1}{32cr})^{|E|-1} \leq 1 - (1 - \frac{|E|-1}{32cr}) \leq \frac{|E|}{32cr}$. Then, by linearity of expectation,  $E[X_E] \leq \frac{|E|^2}{32cr} \leq \frac{|E|}{32}$. We claim that the random variable $X_E$ satisfies the assumptions of Theorem \ref{thm_Talagrand} (Talagrand's Inequality) with $n=|E|$, $a=2$, and $b=2$. The value of $X_E$ is determined by $|E|$ independent trials. Changing the outcome of any trial can affect $X_E$ by at most $2$. For any $s>0$, if it is given that $X_E \geq s$, then there is a set of at most $2s$ trials whose outcomes would ensure that $X_E$ is at least $s$, regardless of the outcomes of the remaining trials. This proves our claim. Let $A = \frac{|E|}{2} + 20a\sqrt{bE[X]}  + 64a^2b = \frac{|E|}{2} + 40\sqrt{2E[X]}  + 512$. Applying Theorem \ref{thm_Talagrand} with $t = |E|/2$, we get 
\begin{eqnarray*}
Pr[|X_E - E[X_E]| > A ] & \leq & 4e^{-\frac{|E|^2/4}{64(E[X_E] + |E|/2)}} \\
 & \leq & 4e^{-\frac{|E|^2}{256(17|E|/32)}}~~(\mbox{since }E[X_E] \leq \frac{|E|}{32}) \\
 & \leq & 4e^{-\frac{|E|}{136}} \\
 & < & 4e^{-\ln (16 \Gamma)}~~(\mbox{since } |E| \geq r \geq \lceil 136 \ln(16 \Gamma)\rceil) \\  
  & = & \frac{1}{4\Gamma}.
\end{eqnarray*}
Since $E[X_E] \leq \frac{|E|}{32}$, we have $E[X_E] + A \leq \frac{17|E|}{32} + 40\sqrt{\frac{2|E|}{32}} + 512 =   \frac{17|E|}{32} + 10\sqrt{|E|} + 512 < |E|$, since it is given that $|E| \geq r \geq 2^{12}$. Thus, $Pr[X_E = |E|] = Pr[X_E - |E| \geq 0] \leq Pr[X_E - (E[X_E] + A) \geq 0] = Pr[(X_E - E[X_E]) \geq A] \leq Pr[|X_E - E[X_E]| \geq A] \leq \frac{1}{4\Gamma}$. 

Let $A_E$ denote the bad event that $X_E = |E|$. From the above calculations, we know that $Pr[A_E] \leq \frac{1}{4\Gamma}$. 
We can apply the Local Lemma (Lemma \ref{lem:local}) on the events $A_E$ for all hyperedges $E \in \mathcal E$.
Since each hyperedge intersects with at most $\Gamma$ other hyperedges, and $4\cdot \frac{1}{4\Gamma}\cdot \Gamma \leq 1$,  
we get $Pr[\cap_{E \in \mathcal{E}}(\overline A_E)] > 0$. That is, there is a conflict free coloring of $\mathcal{H}$ using at most $32cr$ colors. This completes the proof of the lemma. 
\end{proof}

\begin{lemma}
\label{lem_CFON_planar_trick}
Let $G$ be a graph with (i) $V(G) = X \uplus Y$, $X, Y \neq \emptyset$, (ii) every vertex in $G$ has  at most $d_X$ neighbors in $X$, (iii) every vertex in $Y$ has at least one neighbor in $X$, and (iv) every vertex in $X$ has at most $d_Y$ neighbors in $Y$. Then, there is a coloring of the vertices of $X$ with $d_Xd_Y + d_X - d_Y + 1$ colors such that every vertex in $Y$ sees some color exactly once among its neighbors in $X$. 
\end{lemma}
\begin{proof}
For each vertex $y\in Y$, we arbitrarily choose one of its neighbors in $X$. Let us call this neighbor $f(y)$.
For each $y \in Y$, contract the edges $\{y, f(y)\}$ to obtain a resulting graph $G_X$.
Note that the vertex set of $G_X$ is $V(G_X) = X$. The maximum degree of a vertex in the new graph $G_X$ is at most $(d_X - 1)d_Y + d_X$. Thus, we can do a proper coloring (such that no pair of adjacent vertices receive the same color) of $G_X$ using $d_Xd_Y + d_X - d_Y + 1$ colors. We note that this coloring of the vertices of $X$ satisfies our requirement: in the original graph $G$, 
for each $y \in Y$, the neighbor $f(y)$ is colored distinctly from all the other neighbors of $y$ in $X$.
\end{proof}

\begin{lemma}
\label{lem_CFON_hypergraph_degree}
Let $G$ be a graph with (i) $V(G) = X \uplus Y$, $X, Y \neq \emptyset$, (ii) every vertex in $Y$ has at most $t_X$ neighbors in $X$, and (iii) every vertex in $X$ has at least one neighbor in $Y$. Then, there  is a coloring of the vertices of $Y$ using at most $(t_X + 1)$ colors such that every vertex in $X$ sees some color exactly once among its neighbors in $Y$. 
\end{lemma}
\begin{proof}
For every vertex $v \in X$, let $N_G^Y(v)$ denote the set $N_G(v) \cap Y$, i.e., the neighbors of $v$ in $Y$ in the graph $G$. Since every vertex in $X$ has at least one neighbor in $Y$, we have, $|N_G^Y(v)| \geq 1$. We construct a hypergraph $\mathcal{H} = (V, \mathcal{E})$ from $G$ as described below. We have (i) $V = Y$, and (ii) $\mathcal{E} = \{N_G^Y(v)~:~v \in X\}$. Since every vertex in $Y$ has at most $t_X$ neighbors in $X$ in the graph $G$, the maximum degree of a vertex in the hypergraph $\mathcal{H}$ (that is, the maximum number of hyperedges a vertex in $\mathcal{H}$ is part of) is at most $t_X$. From Theorem \ref{thm_pach_tardos}, we have $\chi_{CF}(\mathcal{H}) \leq t_X + 1$. Observe that in this coloring of the vertices of $Y$ using at most $(t_X + 1)$ colors, every vertex in $X$ sees some color exactly once among its neighbors in $Y$. 
\end{proof}


\subsection{Proof of Theorem \ref{thm_main}}
Consider a proper coloring (such that no pair of adjacent vertices receive the same color) of $G$, $h:V(G) \rightarrow [\Delta + 1]$, using $\Delta + 1$ colors. 
Let $C_1, C_2, \ldots , C_{\Delta+1}$ be the color classes given by this coloring $G$. That is, $V(G) = C_1 \uplus C_2 \uplus \cdots \uplus C_{\Delta + 1}$ is the partitioning of the vertex set of $G$ given by the coloring, where each $C_i$ is an independent set.  We may assume
that the coloring $h$ satisfies the following property:  for every $1 < i \leq \Delta + 1$, every vertex $v$ in $C_i$ has at least one neighbor in every $C_j$, where $1 \leq j < i$  (otherwise, we can move $v$ to a color class $C_j$, $j<i$, in which it has no neighbors without compromising on the `properness' of the coloring). Since $G$ is $K_{1,k}$-free, we have the following observation. 
\begin{observation}
\label{obv_degree_in_color_class}
For every $i \in [\Delta + 1]$, a vertex in $G$ has at most $k-1$ neighbors in $C_i$. 
\end{observation}
Let $r = \max(2^{12}, \lceil 136\ln(16\Delta^2) \rceil )$. We partition the vertex set of $G$ into three parts, namely $V_1,V_2,$ and $V_3$ as described below. We have $V_1 := C_1$. If $\Delta > r$, then $V_2 := C_2 \uplus C_3 \uplus \cdots \uplus C_{r+1}$ and $V_3 := C_{r+2} \uplus C_{r+3} \uplus \cdots \uplus C_{\Delta + 1}$. Otherwise, $V_2 := C_2 \uplus C_3 \uplus \cdots \uplus C_{\Delta +1}$ and $V_3 := \emptyset$. 

The rest of the proof is about constructing a coloring $f:V(G)\rightarrow \mathbb{N} \times \mathbb{N}$ that is a CFON coloring of $G$. Let $N_1 = \{1, 2, \ldots , r_1\}$, 
$N_2 = \{r_1+1, r_1+2, \ldots , r_1+ r_2\}$, and 
$N_3 = \{r_1 + r_2+1, r_1 + r_2 +2, \ldots , r_1 + r_2 + r_3\}$, where 
$|N_1| = r_1 =  (k-1)(k-2)r + k$, 
$|N_2| = r_2 = 32(k-1)r$, 
 and $|N_3| = r_3 = k$. 
We define three colorings $f_1, f_2,$ and $f_3$ below.

We begin by describing the coloring $f_1:V_1\rightarrow N_1$. Let $G[V_1 \cup V_2]$ be the subgraph of $G$ induced on $V_1 \cup V_2$.   From Observation \ref{obv_degree_in_color_class}, every vertex in $G[V_1 \cup V_2]$ has at most $k-1$ neighbors in $V_1 = C_1$. Every vertex in $V_2$ has at least one neighbor in $V_1$ due to the property of our coloring $h$. From Observation \ref{obv_degree_in_color_class}, we can also say that every vertex in $V_1$ has at most $r(k-1)$ neighbors in $V_2$. Applying Lemma \ref{lem_CFON_planar_trick} on $G[V_1 \cup V_2]$ with $X = V_1$, $Y = V_2$, $d_X = k-1$ and $d_Y = r(k-1)$, we can say that there is a coloring $f_1:V_1 \rightarrow N_1$ of the vertices of $V_1$ with $(k-1)(k-2)r + k$ colors such that every vertex in $V_2$ sees some color exactly once among its neighbors in $V_1$.

We now describe the coloring $f_2:V_2\rightarrow N_2$. If $V_3 = \emptyset$, then, $\forall v \in V_2$, $f_2(v) = r_1 + 1$. Suppose $V_3 \neq \emptyset$. For a vertex $v$ in $G$, let $N_G^{V_2}(v)$ denote the set of neighbors of $v$ in $V_2$ in the graph $G$. We construct a hypergraph $\mathcal{H}_2 = (V_2, \mathcal{E}_2)$ as follows. We have $\mathcal{E}_2 = \{N_G^{V_2}(v)~:~v \in V_3\}$. Consider an arbitrary hyperedge $E \in \mathcal{E}_2$. In the graph $G$, since every vertex in $V_3$ has at least one neighbor in every color class $C_i$, $2 \leq i \leq r+1$, $|E| \geq r$.  Using Observation \ref{obv_degree_in_color_class}, we can say that $|E| \leq (k-1)r$. As $|N_G^{V_2}(v)| \leq N_G(v) \leq \Delta, \forall v \in V(G)$, we have $|E| \leq \Delta$. This also implies that $E$ intersects with at most $\Delta^2$ other hyperedges in $\mathcal E_2$. Applying Lemma \ref{lem_near_uniform_hypergraph} with $c = (k-1)$ and $\Gamma = \Delta^2$, we have $\chi_{CF}(\mathcal{H}_2) \leq 32(k-1)r$. Thus, there is a coloring $f_2:V_2 \rightarrow N_2$ of the vertices $V_2$ such that  every vertex in $V_3$ sees some color exactly once among its neighbors in $V_2$. 


Finally, we describe the coloring $f_3:V_2\cup V_3 \rightarrow N_3$. From Observation \ref{obv_degree_in_color_class}, every vertex in $V_2 \cup V_3$ has at most $k-1$ neighbors in $V_1 = C_1$. Since there are no isolated vertices in $G$, every vertex in $V_1$ has at least one neighbor in $V_2 \cup V_3$. Applying Lemma \ref{lem_CFON_hypergraph_degree} with $X = V_1$, $Y = V_2 \cup V_3$, and $t_X = k-1$, we get a coloring $f_3:V_2\cup V_3 \rightarrow N_3$ of the vertices of $V_2 \cup V_3$ using at most $k$ colors such that every vertex in $V_1$ sees some color exactly once among its neighbors in $V_2 \cup V_3$.    

We are now ready to define the coloring $f$. 
\begin{align*}
	f(v) = 
	\begin{cases}
	(1, f_1(v)) \text{, if $v \in V_1$}\\
	(f_2(v), f_3(v)) \text{, if $v \in V_2$}\\
	(1, f_3(v)) \text{, if $v \in V_3$}
	\end{cases}\;.
\end{align*}
We now argue that $f$ is indeed a CFON coloring of $G$. Consider a vertex $v \in V(G)$. If  $v \in V_3$, $v$ sees some color exactly once among its neighbors in $V_2$ under the coloring $f_2$. Let $u$ be that neighbor of $v$ in $V_2$ and $f_2(u)$ be that color that appears exactly once in the neighborhood of $v$ in $V_2$. Since the codomains of $f_1$, $f_2$, and $f_3$ are pairwise disjoint sets, $v$ does not see the same color among its neighbors in $V_1$ or in $V_2$. Further, since $f(u) = (f_2(u), f_3(u))$, the final coloring $f$ only refines the color classes of $V_2$ given by $f_2$.  Thus, the color $(f_2(u), f_3(u))$ appears exactly once among the neighbors of $v$ in $G$. The cases when $v \in V_1$ and $v \in V_2$ also follow using similar arguments. 

The coloring $f$ uses at most $|N_1|+ |N_2||N_3| + |N_3| = (k-1)(k-2)r + k + 32(k-1)kr + k$ colors. Since $r = O(\ln \Delta)$, this implies that  
$\chi_{CF}^{ON}(G) = O(k^2\ln\Delta)$.

\bibliographystyle{alpha}
\bibliography{bibfile}

\begin{thebibliography}{ELRS04}

\bibitem[BKM21]{SSR_JGT}
Sriram Bhyravarapu, Subrahmanyam Kalyanasundaram, and Rogers Mathew.
\newblock A short note on conflict-free coloring on closed neighborhoods of
  bounded degree graphs.
\newblock {\em J. Graph Theory}, 97(4):553--556, 2021.

\bibitem[DP21]{DebskiPrzyblo}
Micha\l{} D{\k e}bski and Jakub Przyby\l{}o.
\newblock Conflict-free chromatic number versus conflict-free chromatic index.
\newblock {\em Journal of Graph Theory}, 2021.

\bibitem[EL75]{lovaszlocallemma}
P.~Erd\H{o}s and L.~Lov{\'a}sz.
\newblock Problems and results on 3-chromatic hypergraphs and some related
  questions.
\newblock {\em Infinite and finite sets}, 10:609--627, 1975.

\bibitem[ELRS04]{Even2002}
Guy Even, Zvi Lotker, Dana Ron, and Shakhar Smorodinsky.
\newblock Conflict-free colorings of simple geometric regions with applications
  to frequency assignment in cellular networks.
\newblock {\em SIAM Journal on Computing}, 33(1):94--136, January 2004.

\bibitem[GST14]{glebov2014conflict}
Roman Glebov, Tibor Szab{\'o}, and G{\'a}bor Tardos.
\newblock Conflict-free colouring of graphs.
\newblock {\em Combinatorics, Probability and Computing}, 23(3):434--448, 2014.

\bibitem[MR14]{MolloyR14}
Michael Molloy and Bruce~A. Reed.
\newblock Colouring graphs when the number of colours is almost the maximum
  degree.
\newblock {\em J. Comb. Theory, Ser. {B}}, 109:134--195, 2014.

\bibitem[PT09]{Pach2009}
Janos Pach and G\'{a}bor Tardos.
\newblock Conflict-free colourings of graphs and hypergraphs.
\newblock {\em Combinatorics, Probability and Computing}, 18(5):819--834, 2009.

\bibitem[Smo13]{smorosurvey}
Shakhar Smorodinsky.
\newblock {\em Conflict-Free Coloring and its Applications}, pages 331--389.
\newblock Springer Berlin Heidelberg, Berlin, Heidelberg, 2013.

\end{thebibliography}

\end{document}